\documentclass[elsarticle]{article}
\usepackage{latexsym}
\usepackage{amsfonts}
\usepackage{amssymb}
\usepackage{amscd}
\usepackage{amsmath}
\usepackage{amsthm}
\usepackage{url}
\usepackage[hmargin=1.1in,vmargin=1.1in]{geometry}
\usepackage{graphicx}
\usepackage{mathtools}

\usepackage[T1]{fontenc}
\usepackage[latin1]{inputenc}
\usepackage[compact]{titlesec}
\usepackage{sectsty}
\allsectionsfont{\scshape\selectfont}

\theoremstyle{plain}
\newtheorem{thm}{Theorem}[section]
\theoremstyle{definition}		
\newtheorem{definition}[thm]{Definition}

\newcommand{\cF}{\mathcal{F}}
\newcommand{\cL}{\mathcal{L}}

\newcommand{\CC}{\mathbb{C}}
\newcommand{\RR}{\mathbb{R}}
\newcommand{\EE}{\mathbb{E}}
\newcommand{\p}{\rm {I\kern-1pt P}}

 \newtheorem{lemma}{Lemma}[section]
\newtheorem{proposition}{Proposition}[section]
\newtheorem{theorem}{Theorem}[section]
\newcommand{\MT}{\left[ \begin{array}{ccccccccc}}
\newcommand{\EM}{\end{array}\right]}

\newtheorem{assumption}{Assumption}[section]

\newcommand{\LM}{\left[\begin{array}{ccccccccc}}
\newcommand{\RM}{\end{array}\right]}
\newcommand{\LA}{\left\{ \begin{array}{ccccccccc}}
\newcommand{\RA}{\end{array}\right.}
\newcommand{\RAA}{\end{array}\right\}}

\newcommand{\lin}{\text{lin}}
\newcommand{\Sx}{S_{\mathbf{x}}}
\newcommand{\Sxs}{S_{\mathbf{x}}^{\ast}}

\newcommand{\bPsi}{\boldsymbol{\Psi}}
\newcommand{\bPhi}{\boldsymbol{\Phi}}

\newcommand{\bx}{\mathbf{x}}

\newcommand{\cH}{\mathcal{H}}
\newcommand{\cR}{\mathcal{R}}
\newcommand{\cO}{\mathcal{O}}
\newcommand{\bbE}{\mathbb{E}}
\newcommand{\bbP}{\mathbb{P}}
\newcommand{\cX}{\mathcal{X}}

\newcommand{\tr}{\!\top\!}
\newcommand{\scal}[2]{\left\langle{#1},{#2}\right\rangle}
\newcommand{\nor}[1]{\left\|{#1}\right\|}

\def\nonum{\nonumber}
\def\[{\begin{equation}}
\def\]{\end{equation}}
\def\NN{\mbox{I}\!\mbox{N}}

\def\nonum{\nonumber}

\def\[{\begin{equation}}
\def\]{\end{equation}}

\def\nonum{\nonumber}

\def\keywords#1{\gdef\@keywords{#1}}
\gdef\@keywords{}
\title{\LARGE \bf\scshape Kernel Methods for the Approximation of Some Key Quantities of Nonlinear Systems\footnote{Parts of this work were done while the authors were at  the Department of Mathematics of Duke University and then while the second author was with the Department of Mathematics of Imperial College London for a Marie Curie Fellowship and at the Fields Institute.}}

\author{Jake Bouvrie
\thanks{J. Bouvrie is with the Laboratory for Computational and Statistical Learning, 
Massachusetts Institute of Technology, Cambridge, MA, USA  {\tt\small jvb@csail.mit.edu}}
\and
Boumediene Hamzi
\thanks{B. Hamzi is with the Department of Mathematics, Ko\c{c} University, Istanbul, Turkey {\tt\small boumediene.hamzi@gmail.com}}%
}

\date{}

\begin{document}

\maketitle

\begin{abstract}
We introduce a data-based approach to estimating key quantities which arise in the
study of nonlinear control systems and random nonlinear dynamical systems. Our approach
hinges on the observation that much of the existing linear theory may be readily extended
to nonlinear systems -- with a reasonable expectation of success -- once the nonlinear system
has been mapped into a high or infinite dimensional feature space. In particular, we embed a nonlinear system in a reproducing kernel Hilbert space where linear theory can be used to  develop
computable, non-parametric estimators approximating controllability and observability energy functions for
nonlinear systems.
In all cases the relevant quantities are estimated from simulated or observed data.
It is then shown that the controllability energy estimator provides a key means for approximating
the invariant measure of an ergodic, stochastically forced nonlinear system.
\end{abstract}


%


\section{Introduction}
Personal computing has developed to the point where in many cases it ought to be easier to simulate a dynamical
system and analyze the empirical data, rather than attempt to study the system analytically. Indeed,
for large classes of nonlinear systems, numerical analysis may be the only viable option. Yet the
mathematical theory necessary to analyze dynamical systems on the basis of observed data is still
largely underdeveloped. In previous work the authors proposed a linear, data-based approach for model reduction
of nonlinear control systems~\cite{allerton}. The approach is based on lifting simulated trajectories
of the system into a high or infinite dimensional feature (Hilbert) space where the evolution of the
original system may be reasonably modeled as linear. One may then implicitly carry out linear balancing,
truncation, and model reduction in the feature space while retaining nonlinearity in the original
statespace.

In this paper, we continue under this setting and explore data-based definitions of key concepts for
nonlinear control and random dynamical systems. We propose an empirical approach for
estimation of the controllability and observability energies for stable nonlinear control systems, as well
as invariant  measures  for ergodic nonlinear stochastic differential
equations. Our methodology applies the relevant linear theory in a feature space where it is assumed
that the original nonlinear system behaves approximately linearly. In this case we leverage the well-known
connection between the controllability gramian of a linear control system and the invariant measure
associated to the corresponding linear stochastic differential equation. This relationship was previously
identified as useful for finding the controllability energy for certain nonlinear control systems given the
invariant measure of the corresponding randomly forced dynamical system~\cite{newman}. The approach
in~\cite{newman}, however, requires solving a Fokker-Planck equation and so applies to only a narrow class
of systems. Our approach takes the reverse direction in a data-driven setting: given an empirical estimate
of the controllability energy function, one can obtain an estimate of the invariant measure. In particular,
we will propose a consistent, data-based estimator for the controllability energy function of a
nonlinear control system, and show how it can be used to estimate the invariant measure for the corresponding stochastic differential equation (SDE).

The essential point of this paper is to illustrate that it is possible to find data-based estimates of nonlinear objects that allow  to characterize the qualitative behaviour of nonlinear control and random dynamical systems, without having to
solve a Hamilton-Jacobi-Bellman or Lyapunov equation in the case of nonlinear control systems, or a
Fokker-Planck equation in the case of nonlinear SDEs. The approach proposed here also highlights the
close interaction between control and random dynamical systems and  demonstrates how control
theoretic objects can be useful for studying random dynamical systems.
Our contribution should be seen as a step towards developing a mathematical, data-driven qualitative theory of dynamical systems which can be used to analyze and predict random dynamical systems, as well as offer
data-driven control strategies for nonlinear systems on the basis of observed data
rather than a pre-specified model.

Preliminary results of this work can be found in \cite{allerton, acc2012}.


\section{Linear Systems as a Paradigm for Working in RKHS: Background}\label{sec:bkgnd}
In this section we give a brief overview of some important background concepts in linear control,
random dynamical systems and reproducing kernel Hilbert spaces (RKHS). We will make use of the linear theory that follows after mapping the state variable of a nonlinear system into a suitable RKHS, thereby
harnessing RKHS theory as a framework for extending linear tools to nonlinear systems. The following
background material closely follows~\cite{dullerud,risken,brockett}.

\subsection{Linear Control Systems}\label{sec:linear-control}
Consider a linear control system
\[\label{linsys}\begin{array}{rcl}\dot{x}&=&Ax+Bu\\y&=&Cx \end{array}, \]
where $x\in\RR^n$, $u \in \RR^q$, $y \in \RR^p$, $(A,B)$ is controllable, $(A,C)$ is observable
and $A$ is Hurwitz. We define the controllability and the
observability Gramians as, respectively,
$W_c=\int_0^{\infty}e^{At}BB^{\tr}e^{A^{\tr}t}\, dt,$
$W_o=\int_0^{\infty}e^{A^{\tr}t}C^{\tr}Ce^{At}\, dt$.
These two matrices can be viewed as a measure of the controllability and the observability of the
system~\cite{moore}. For instance,
 consider the past energy~\cite{scherpen_thesis}, $L_c(x_0)$, defined
as the minimal energy required to reach $x_0$ from $0$ in infinite time
\[\label{L_c}
L_c(x_0)=\inf_{\substack{
u \in { L}_2(-\infty,0),\\ x(-\infty)=0, x(0)=x_0}}
\frac{1}{2}\int_{-\infty}^0\|u(t)\|^2\,dt,
\]
 and  the future energy~\cite{scherpen_thesis}, $L_o(x_0)$, defined as the output energy generated
by releasing the system from its initial state $x(t_0)=x_0$, and zero input $u(t)=0$ for $t\ge0$, i.e.
 \[\label{L_o}
 L_o(x_0)=\frac{1}{2}\int_{0}^{\infty}\|y(t)\|^2\,dt,
 \]
for $x(t_0)=x_0$ and $u(t)=0, t\ge0$.

In the linear case, it can be shown that
\begin{align}
L_c(x_0)&=\tfrac{1}{2}x_0^{\tr}W_c^{-1}x_0,\label{eqn:lin_Lc}\\
L_o(x_0)&=\tfrac{1}{2}x_0^{\tr}W_o x_0 \label{eqn:lin_Lo} .
\end{align}
Moreover, $W_c$ and $W_o$ satisfy the following Lyapunov equations~\cite{dullerud}:
\[\label{lyap_lin}\begin{array}{rcl}
AW_c+W_cA^{\tr}&=&-BB^{\tr},\\
A^{\tr}W_o+W_oA&=&-C^{\tr}C .
\end{array}
\]
These
energies are directly related to the controllability and observability operators.

\begin{definition}\cite{dullerud}
 Given a matrix pair $(A,B)$, the controllability operator $\Psi_c$ is defined as
\[\begin{array}{rcl} \Psi_c: L_2(-\infty,0)
&\rightarrow& \CC^n \\
 u &\mapsto& \int_{-\infty}^0 e^{-A\tau}Bu(\tau) d\tau \nonumber
\end{array}
\]
\end{definition}

The significance of this operator is made evident via the following optimal control problem: Given
the linear system $\dot{x}(t)=Ax(t)+Bu(t)$ defined for $t \in (-\infty,0)$ with $x(-\infty)=0$, and
for
 $x(0) \in \CC^n$ with unit norm, what is the minimum energy
input $u$ which drives the state $x(t)$ to $x(0)=x_0$ at time zero? That is, what is the
$u \in L_2(-\infty,0]$ solving $\Psi_cu=x_0$ with smallest norm $\|u\|_2$?
If $(A,B)$ is controllable, then $\Psi_c\Psi_c^{\ast}=:W_c$ is nonsingular, and the answer to the preceding
question is \[\label{uopt} u_{opt}:=\Psi_c^{\ast}W_c^{-1}x_0. \] The input energy is given by
\begin{eqnarray}
\|u_{opt}\|_2^2 &=& \langle \Psi_c^{\ast}W_c^{-1}x_0, \Psi_c^{\ast}W_c^{-1}x_0 \rangle \nonum \\
&=& \langle W_c^{-1}x_0, \Psi_c\Psi_c^{\ast}W_c^{-1}x_0 \rangle \nonum \\
&=& x_0^{\ast}W_c^{-1}x_0 \;.\nonum
\end{eqnarray}

Moreover, the reachable set through $u_{opt}$, i.e. the final states $x_0=\Psi_c u$ that can be reached given
an input $u \in L_2(-\infty,0]$ of unit norm,  $\{\Psi_cu: u \in L_2(-\infty,0] \;\mbox{and} \; \|u\|_2 \le 1
\}$ may be defined as
\[
\cR :=\{W_c^{\frac{1}{2}}x_c: x_c \in \CC^n \; \mbox{and} \; \|x_c\|\le 1 \}.
\]


Similarly, for the autonomous system
\[\begin{array}{rcl}
 \dot{x}&=& Ax, \quad x(0)=x_0\in \CC^n,\\
y&=&Cx \nonum
\end{array}
\]
where $A$ is Hurwitz, the observability operator is defined as follows.

\begin{definition} \cite{dullerud}
  Given a matrix pair $(A,C)$, where $A$ is Hurwitz, the observability operator $\Psi_o$ is defined as
\[
\begin{array}{rcl} \Psi_o: \CC^n
&\rightarrow& L_2(0,\infty) \\
 x_0 &\mapsto& \left\{
 \begin{array}{rcl}
 Ce^{At}x_0, & \text{for }  t \ge 0\\ 0, & \text{otherwise}
\end{array}\right.
\end{array}\nonum
\]
The corresponding observability ellipsoid is given by
$$
{\cal E}:=\{W_o^{\frac{1}{2}}x_0: x_0 \in \CC^n \; \mbox{and} \; \|x_0\|=1 \}.
$$
\end{definition}

The energy of the output signal $y=\Psi_o x_0 $, for $x_0 \in \CC^n$ can then be computed as
$$
\|y\|_2^2=\langle \Psi_ox_0, \Psi_ox_0
\rangle=\langle x_0, \Psi_o^{\ast}\Psi_ox_0
\rangle=\langle x_0, W_0x_0
\rangle
$$
where $\Psi_o^{\ast}: L_2[0,\infty) \rightarrow \CC^n $ is the adjoint of $\Psi_o$.

\subsection{Linear Stochastic Differential Equations}\label{sec:linear-sdes}
In this section, we review the relevant background for stochastically forced differential equations
 (see e.g.~\cite{risken,brockett} for more detail). Here we will consider stochastically excited dynamical control systems affine in the input $u\in\RR^q$
\[\label{control_nonlin}
\dot{x}=f(x)+G(x)u \;,
\]
where $G:\RR^n\to\RR^{n\times q}$ is a smooth matrix-valued function and $x\in\RR^n$. We
replace the control inputs by sample paths of white Gaussian noise processes, giving the corresponding stochastic differential equation
 \[\label{sde_nonlin}
 d{X_t}=f(X_t)dt + G(X_t) \circ dW_t^{(q)}
 \]
with $W_t^{(q)}$ a $q-$dimensional Brownian motion. The solution $X_t$ to this SDE is a Markov stochastic
 process with transition probability density $\rho(t,x)$. The time evolution of the probability density
$\rho(t,x)$
is described by the \emph{Fokker-Planck (or Forward Kolmogorov) equation}
\[\label{ito}\frac{\partial \rho}{\partial t}=-\langle \frac{\partial}{\partial x}, f\rho \rangle+
\frac{1}{2}\sum_{j,k=1}^{n}\frac{\partial^2}{\partial x_j \partial x_k}[(GG^{\tr})_{jk}\rho]=:\cL\rho  \;.
\]
The differential operator $\cL$ on the right-hand side is referred to as the \emph{Fokker-Planck operator}
associated to (\ref{sde_nonlin}). The \emph{steady-state probability density} for (\ref{sde_nonlin}) is a
solution of the equation \[\label{steady}\cL{\rho_{\infty}}(x)=0.\]




In the context of linear Gaussian theory where we are given an $n-$dimensional system of the
form \[\label{lin_sde}dX_t=AX_tdt+BdW_t^{(q)},\]
with $A \in \RR^{n \times n}$, $B \in \RR^{n \times q}$, the transition density is Gaussian.
It is therefore sufficient to find the mean and covariance of the solution $X(t)$ in order to uniquely determine the transition probability density. The mean satisfies
$\frac{d}{dt}{\EE}[X]=A {\EE}[X]$ and thus ${\EE}[X(t)]=e^{At}{\EE}[X(0)]$. If $A$
is Hurwitz, $\lim_{t \rightarrow \infty}{\EE}[X(t)]=0$.
The covariance satisfies $\frac{d}{dt}{\EE}[XX^{\tr}]=A{\EE}[XX^{\tr}]+{\EE}[XX^{\tr}]A
+BB^{\tr}$. This formulation gives the steady-state distribution's covariance matrix as ${\cal Q}=\lim_{t \rightarrow \infty}{\EE}[X_tX_t^{\tr}]$ so that we may find ${\cal Q}$ by solving the Lyapunov system $A{\cal Q}+{\cal Q}A^{\tr}=-BB^{\tr}$. Thus the solution ${\cal Q}$ is exactly the controllability
gramian
$$
{\cal Q}=W_c=\int_0^{\infty}e^{At}BB^{\tr}e^{A^{\tr}t}\, dt
$$
which is positive iff. the pair $(A,B)$ is controllable~\cite{brockett}. Combining the above facts, the steady-state probability density is
 given by
\[\label{eqn:rho_linear_theory}
\rho_{\infty}(x)=Z^{-1}e^{-\frac{1}{2}x^TW_c^{-1}x} = Z^{-1} e^{-L_c(x)}
\]
using (\ref{L_c}) and letting $Z=\sqrt{(2\pi)^n \mbox{det}(W_c)}$.

Equation~\eqref{eqn:rho_linear_theory} suggests the following key observations in the linear setting:
\begin{itemize}
\item {\em Given an approximation $\hat{L}_c$ of $L_c$ we obtain an approximation for $\rho_{\infty}$ of the form}
\[\label{rho_approx}
\hat{\rho}_{\infty}(x) \propto e^{-\hat{L}_c(x)}
\]
\item {\em Given an approximation $\hat{\rho}_{\infty}$ of $\rho_{\infty}$ we obtain an approximation for
$L_c(x)$ by solving}
\[\label{Lc_approx}
\hat{L}_c(x)= -\ln[\hat{\rho}_{\infty}(x)]+C .
\]
\end{itemize}

We note that these approximations have been used in different contexts to study nonlinear control and random dynamical systems. For instance, in \cite{butchart}, Equation~\eqref{rho_approx} was used to find explicit solutions of the Fokker-Planck equation for systems where a Lyapunov equation for the unforced system can be found and solved. In~\cite{newman}, Equation~\eqref{Lc_approx} was used, given an explicit solution to the Fokker-Planck equation, to approximate the controllability energy and subsequently applied to the problem of model reduction for nonlinear control systems.

Although the above relationship between $\rho_{\infty}$ and $L_c$ holds for only a small class of systems (e.g. linear and some Hamiltonian systems), by mapping a nonlinear system into a suitable reproducing kernel Hilbert space we may reasonably extend this connection to a broad class of nonlinear systems. We will return to this topic in Section~\ref{sec:nonlinear-sdes} after defining kernel Hilbert spaces and introducing gramians in RKHS.


\subsection{Reproducing Kernel Hilbert Spaces}
We give a brief overview of reproducing kernel Hilbert spaces as used in statistical learning
theory. The discussion here borrows heavily from~\cite{cucker,smola,Wahba}. Early work developing
the theory of RKHS was undertaken by N. Aronszajn~\cite{AronRKHS}.

\begin{definition} Let  ${\cal H}$  be a Hilbert space of functions on a set ${\cal X}$.
Denote by $\langle f, g \rangle$ the inner product on ${\cal H}$   and let $\|f\|= \langle f, f \rangle^{1/2}$
be the norm in ${\cal H}$, for $f$ and $g \in {\cal H}$. We say that ${\cal H}$ is a reproducing kernel
Hilbert space (RKHS) if there exists a function $K:{\cal X} \times {\cal X} \rightarrow \RR$
such that
\begin{itemize}
 \item[i.] $K_x:=K(x,\cdot)\in\cH$ for all $x\in\cX$.
\item[ii.] $K$ spans ${\cal H}$: ${\cal H}=\overline{\mbox{span}\{K_x~|~x \in {\cal X}\}}$.
 \item[iii.] $K$ has the {\em reproducing property}:
$\forall f \in {\cal H}$, $f(x)=\langle f,K_x \rangle$.

\end{itemize}
$K$ will be called a reproducing kernel of ${\cal H}$. ${\cal H}_K$  will denote the RKHS ${\cal H}$
with reproducing kernel  $K$ where it is convenient to explicitly note this dependence.
\end{definition}

\begin{definition}[Mercer kernels]
A function $K:{\cal X} \times {\cal X} \rightarrow \RR$ is called a Mercer kernel if it is continuous,
symmetric and positive definite.
\end{definition}

The important properties of reproducing kernels are summarized in the following proposition.
\begin{proposition}\label{prop1} If $K$ is a reproducing kernel of a Hilbert space ${\cal H}$, then
\begin{itemize}
\item[i.] $K(x,y)$ is unique.
\item[ii.]  $\forall x,y \in {\cal X}$, $K(x,y)=K(y,x)$ (symmetry).
\item[iii.] $\sum_{i,j=1}^q\alpha_i\alpha_jK(x_i,x_j) \ge 0$ for $\alpha_i \in \RR$, $x_i \in {\cal X}$ and $q\in\mathbb{N}_+$
(positive definiteness).
\item[iv.] $\langle K(x,\cdot),K(y,\cdot) \rangle=K(x,y)$.
\end{itemize}
\end{proposition}
Common examples of Mercer kernels defined on a compact domain $\cX\subset\RR^n$ are
$K(x,y)=x\cdot y$ (Linear), $K(x,y)=(1+x\cdot y)^d$ for $d \in \NN_+$ (Polynomial), and
 $K(x,y)=e^{-\|x-y\|^2_2/\sigma^2}, \sigma >0$ (Gaussian).

\begin{theorem} \label{thm1}
Let $K:{\cal X} \times {\cal X} \rightarrow \RR$ be a symmetric and positive definite function. Then there
exists a Hilbert space of functions ${\cal H}$ defined on ${\cal X}$   admitting $K$ as a reproducing Kernel.
Conversely, let  ${\cal H}$ be a Hilbert space of functions $f: {\cal X} \rightarrow \RR$ satisfying
$\forall x \in {\cal X}, \exists \kappa_x>0,$ such that $|f(x)| \le \kappa_x \|f\|_{\cal H},
\quad \forall f \in {\cal H}. $
Then ${\cal H}$ has a reproducing kernel $K$.
\end{theorem}


\begin{theorem}\label{thm4}
 Let $K(x,y)$ be a positive definite kernel on a compact domain or a manifold $X$. Then there exists a Hilbert
space $\mathcal{F}$  and a function $\Phi: X \rightarrow \mathcal{F}$ such that
$$K(x,y)= \langle \Phi(x), \Phi(y) \rangle_{\cF} \quad \mbox{for} \quad x,y \in X.$$
 $\Phi$ is called a feature map, and $\mathcal{F}$ a feature space\footnote{The dimension of the feature space can be infinite, for example in the case of the Gaussian kernel.}.
\end{theorem}

Given Theorem~\ref{thm4}, and property [iv.] in Proposition~\ref{prop1}, note that we can take
$\Phi(x):=K_x:=K(x,\cdot)$ in which case $\mathcal{F}=\cH$ -- the ``feature space'' is the RKHS itself, as opposed to an isomorphic space.
We will make extensive use of this feature map. The fact that Mercer kernels are positive definite and
symmetric is also key; these properties ensure that kernels induce positive, symmetric matrices and
integral operators, reminiscent of similar properties enjoyed by gramians and covariance matrices.
Finally,  in practice one typically first chooses a Mercer kernel in order to choose an RKHS:
Theorem~\ref{thm1} guarantees the existence of a Hilbert space admitting such a function as its
reproducing kernel.

A key observation however, is that working in RKHS allows one to immediately find nonlinear versions of
 algorithms which can be expressed in terms of inner products. Consider an algorithm expressed in
terms of the inner product $\langle x, x^{\prime} \rangle_{\cX}$ with $x, x^{\prime} \in \cX$. Now assume that
instead of looking at a state $x$, we look at its $\Phi$ image in $\cH$,
\[\label{eqn:phi-mapped-data}
\begin{array}{rccl}
\Phi&:& X & \rightarrow {\cal H} \\
& & x &\mapsto \Phi(x) \;.
\end{array}\]
In the RKHS, the inner product $\langle \Phi(x), \Phi(x^{\prime}) \rangle$ is
\[\langle \Phi(x), \Phi(x^{\prime}) \rangle= K(x,x^{\prime}) \]
by the reproducing property. Hence, a nonlinear variant of the original algorithm may be implemented
using kernels in place of inner products on $\cX$.

\section{Empirical Gramians in RKHS}
In this Section we recall empirical gramians for linear systems~\cite{moore}, as well as a notion of empirical gramians for nonlinear systems in RKHS introduced in~\cite{allerton}. The goal of the construction we describe here is to provide meaningful, data-based empirical controllability and observability gramians for nonlinear systems. In~\cite{allerton}, observability and controllability gramians were used for balanced model reduction, however here we will use these quantities to analyze nonlinear control properties and random dynamical systems. We note that a related notion of gramians for nonlinear systems is briefly discussed in~\cite{gray}, however no method for computing or estimating them was given.

\subsection{Empirical Gramians for Linear Systems}\label{sec:linear_gramians}
To compute the Gramians for the linear system (\ref{linsys}), one can attempt to solve the
Lyapunov equations (\ref{lyap_lin}) directly although this can be computationally prohibitive. For linear systems, the gramians may be approximated by way of matrix multiplications implementing primal and adjoint systems (see the method of snapshots, e.g.~\cite{Rowley05}). Alternatively, for any system, linear or nonlinear, one may take the simulation based approach introduced by B.C. Moore~\cite{moore} for reduction of linear systems, and subsequently extended to nonlinear systems in~\cite{lall}. The method
proceeds by exciting each coordinate of the input with impulses from the zero initial state $(x_0=0)$. The system's responses are sampled, and the sample covariance is taken as an approximation to the controllability gramian. Denote the set of canonical orthonormal basis vectors in $\RR^n$ by $\{e_i\}_{i}$. Let $u^i(t) = \delta(t)e_i$ be the input signal for the $i$-th simulation, and let $x^i(t)$ be the corresponding response of the system. Form the matrix $X(t) = \bigl[x^1(t)
~\cdots~ x^q(t)\bigr] \in \RR^{n\times q}$, so that $X(t)$ is seen as a data matrix with
column observations given by the respective responses $x^i(t)$. Then the $(n\times n)$ controllability gramian  is given by
\[
W_{c,\lin} = \frac{1}{q}\int_0^{\infty}X(t)X(t)^{\tr} dt.
\]
We can approximate this integral by sampling the matrix function $X(t)$ within a finite time interval $[0,T]$
assuming for instance the regular partition $\{t_i\}_{i=1}^N, t_i = (T/N)i$. This leads to the {\em empirical controllability gramian}
\[\label{eqn:Wchat_lin}
\widehat{W}_{c,\lin} = \frac{T}{Nq}\sum_{i=1}^N X(t_i)X(t_i)^{\tr} .
\]

The observability gramian is estimated by
fixing $u(t) = 0$, setting $x_0 = e_i$ for $i=1,\ldots,n$, and measuring the corresponding system output
responses $y^i(t)$. Now assemble the output responses into a matrix $Y(t) = [y^1(t) ~\cdots~ y^n(t)]\in
\RR^{p\times n}$. The $(n\times n)$ observability gramian $W_{o,\lin}$ and its empirical
counterpart $\widehat{W}_{o,\lin}$ are respectively given by
\[
W_{o,\lin} = \frac{1}{p}\int_0^{\infty}Y(t)^{\tr}Y(t) dt\]
and
\[\label{eqn:Wohat_lin}
\widehat{W}_{o,\lin} = \frac{T}{Np}\sum_{i=1}^N \widetilde{Y}(t_i)\widetilde{Y}(t_i)^{\tr}
\]
where $\widetilde{Y}(t) = Y(t)^{\tr}$.
The matrix $\widetilde{Y}(t_i)\in\RR^{n\times p}$ can be thought of as a data matrix with column observations
\begin{equation}\label{eqn:obs_data}
d_j(t_i) = \bigl(y_j^1(t_i), \ldots, y_j^n(t_i)\bigr)^{\!\tr} \in\RR^n,
\end{equation}
for $j=1,\ldots,p, \,\,i=1,\ldots, N$ so that $d_j(t_i)$ corresponds to the response at time $t_i$ of the
single output coordinate $j$ to each of
the (separate) initial conditions $x_0=e_k, k=1,\ldots,n$.

\subsection{Empirical Gramians for Nonlinear Systems in RKHS}\label{sec:rkhs-gramians}
Consider the generic nonlinear system
\[\label{sigma}
\left\{\begin{array}{rcl}\dot{x}&=&F(x,u)\\ y &=& h(x), \end{array}\right.
\]
with $x \in \RR^n$, $u \in \RR^q$, $y\in \RR^p$, $F(0)=0$ and $h(0)=0$.
Assume that  the linearization of~\eqref{sigma} around the origin is controllable, observable and
$A=\frac{\partial F}{\partial x}|_{x=0}$ is asymptotically stable.

RKHS counterparts to the empirical quantities~\eqref{eqn:Wchat_lin},\eqref{eqn:Wohat_lin} defined above for the system~\eqref{sigma} can be defined by considering feature-mapped lifts of the simulated samples in $\cH_K$. In the following, and without loss of generality,
{\em we assume the data are centered in feature space}, and that the observability
samples and controllability samples are centered separately. See~(\cite{smola}, Ch. 14) for a
discussion on implicit data centering in RKHS with kernels.

First, observe that the gramians $\widehat{W}_c, \widehat{W}_o$ can be viewed as the sample covariance of a collection of $N\cdot q, N\cdot p$ vectors in $\RR^n$ scaled by $T$, respectively. Then applying $\Phi$ to the samples as in~\eqref{eqn:phi-mapped-data}, we obtain the corresponding gramians in the RKHS associated to $K$ as bounded linear operators on $\cH_K$:
\begin{align}
\widehat{W}_c &= \frac{T}{Nq}\sum_{i=1}^N\sum_{j=1}^q \Phi(x^j(t_i))\otimes \Phi(x^j(t_i)) \label{eqn:emp_Wc_rkhs}\\
\widehat{W}_o &= \frac{T}{Np}\sum_{i=1}^N\sum_{j=1}^p \Phi(d_j(t_i))\otimes\Phi(d_j(t_i))\nonumber
\end{align}
where the samples $x_j,d_j$ are as defined in Section~\ref{sec:linear_gramians}, and $a\otimes b=a\scal{b}{\cdot}$ denotes the tensor product in $\cH$. From here on we will use the notation $W_c, W_o$ to refer to RKHS versions of
the true (integrated) gramians, and $\widehat{W}_c, \widehat{W}_o$ to refer to RKHS versions of the empirical gramians.

Let $\bPsi$ denote the matrix whose columns are the (scaled) observability samples mapped into feature space by $\Phi$, and let $\bPhi$ be the matrix similarly built from the feature space representation of the
controllability samples. Then we may alternatively express the gramians above as
$\widehat{W}_c=\bPhi\bPhi^{\tr}$ and $\widehat{W}_o=\bPsi\bPsi^{\tr}$, and define two other important quantities:
\begin{itemize}\itemsep 0pt
\item The \emph{controllability kernel matrix} $K_c\in\RR^{Nq\times Nq}$ of kernel
products
\begin{align}
K_c &= \bPhi^{\tr}\bPhi \\
(K_c)_{\mu\nu} &= K(x_\mu, x_\nu) = \scal{\Phi(x_\mu)}{\Phi(x_\nu)}_{\cF}
\end{align}
for $\mu,\nu=1,\ldots,Nq$ where we have re-indexed the set of vectors  $\{x^{j}(t_i)\}_{i,j} =
\{x_{\mu}\}_{\mu}$ to use a single linear index.
\item The \emph{observability kernel matrix} $K_o\in\RR^{Np\times Np}$,
\begin{align}
K_o &= \bPsi^{\tr}\bPsi\\
(K_o)_{\mu\nu} &= K(d_\mu, d_\nu) = \scal{\Phi(d_\mu)}{\Phi(d_\nu)}_{\cF}
\end{align}
for $\mu,\nu=1,\ldots,Np$, where we have again re-indexed the set $\{d_j(t_i)\}_{i,j}=\{d_\mu\}_{\mu}$ for
simplicity.
\item  The \emph{Hankel kernel matrix} $K_{o,c}\in\RR^{Np\times Nm}$,
\begin{align} K_{o,c} &= \bPsi^{\tr}\bPhi \\ (K_{o,c})_{\mu\nu} & = K(d_\mu, x_\nu) = \scal{\Phi(d_\mu)}{\Phi(x_\nu)}_{\cal H} \end{align}
for $\mu=1,\ldots,Np$, $\nu=1,\ldots,Nm$.

\end{itemize}
Note that $K_c,K_o, K_{o,c} $ may be highly ill-conditioned. The SVD
may be used to show that $\widehat{W}_c$ and $K_c$ ($\widehat{W}_o$ and $K_o$) have the same singular values (up to zeros).

\section{Nonlinear Control Systems in RKHS}
In this section, we introduce empirical versions of the controllability and observability energies~\eqref{L_c}-\eqref{L_o} for stable
nonlinear control systems of the form~\eqref{control_nonlin}, that can be estimated from observed data. Our underlying assumption is that a given nonlinear system may be treated as if it were linear in a suitable feature space. That reproducing kernel Hilbert spaces provide rich representations capable of capturing strong nonlinearities in the original input (data) space lends validity to this assumption.

In general little is known about the energy functions in the nonlinear setting. However,
Scherpen~\cite{scherpen_thesis} has shown that the energy functions $L_c(x)$ and $L_o(x)$ defined in~\eqref{L_c} and~\eqref{L_o} satisfy a Hamilton-Jacobi and a Lyapunov equation, respectively.
\begin{theorem}\label{thm:scherp1}\cite{scherpen_thesis} Consider the nonlinear control system (\ref{sigma}) with $F(x,u)=f(x)+G(x)u$. If the origin is an asymptotically stable equilibrium
of $f(x)$ on a neighborhood $W$ of the origin, then for
all $x \in W$, $L_o(x)$ is the unique smooth solution of
\[\label{Lo_hjb} \frac{\partial L_o}{\partial x}(x)f(x)+\frac{1}{2}h^{\tr}(x)h(x)=0,\quad L_o(0)=0 \]
under the assumption that (\ref{Lo_hjb}) has a smooth solution on $W$. Furthermore for all $x \in W$, $L_c(x)$
is the unique smooth solution of
\[\label{Lc_hjb} \frac{\partial L_c}{\partial x}(x)f(x)+\frac{1}{2} \frac{\partial L_c}{\partial
x}(x)g(x)g^{\tr}(x)  \frac{\partial^{\tr}L_c}{\partial x}(x)=0,\; L_c(0)=0\]
under the assumption that (\ref{Lc_hjb}) has a smooth solution $\bar{L}_c$ on $W$ and that the origin is an
asymptotically stable equilibrium of $-(f(x)+g(x)g^{\tr}(x) \frac{\partial \bar{L}_c}{\partial x}(x))$ on $W$.
\end{theorem}

 Various methods have been proposed to find approximate solutions of the PDEs the PDEs (\ref{Lo_hjb})- (\ref{Lc_hjb}).   For instance, Taylor series expansions have been proposed in \cite{krener1,krener2,fujimoto}. Newman and Krishnaprasad~\cite{newman} introduce a statistical approximation based on exciting the system with white Gaussian noise and then computing the balancing transformation using an algorithm from differential topology. 
We would like to avoid solving explicitly the PDEs (\ref{Lo_hjb})- (\ref{Lc_hjb}) and instead find good
estimates of their solutions directly from simulated or observed data.

\subsection{Energy Functions}\label{sec:energy_fns}
Following the linear theory developed in Section~\ref{sec:linear-control}, we would like to
define analogous controllability and observability energy functions paralleling~\eqref{eqn:lin_Lc}-\eqref{eqn:lin_Lo}, but adapted to the nonlinear setting. We first treat the controllability
function. Let $\mu_{\infty}$ on the statespace $\cX$ denote the unknown invariant measure of the nonlinear system~\eqref{sigma} when driven by white Gaussian noise. We will consider here the case where the controllability samples $\{x_i\}_{i=1}^m$ are i.i.d. random draws from $\mu_{\infty}$, and $\cX$ is a compact subset of $\RR^n$. The former assumption is implicitly made in much of the empirical balancing literature, and if a system is simulated for long time intervals, it should hold approximately in practice. If we take $\Phi(x)=K_x$, the infinite-data limit of~\eqref{eqn:emp_Wc_rkhs} is given by
\[\label{eqn:covop_gramian}
W_c = \bbE_{\mu_{\infty}} [\widehat{W}_{c}] = \int_{\cX}\scal{\cdot}{K_x}K_xd\mu_{\infty}(x).
\]

In general neither $W_c$ nor its empirical approximation $\widehat{W}_c$ are invertible, so to define a controllability energy similar to~\eqref{eqn:lin_Lc} one is tempted to define $L_c$ on $\cH$ as
\mbox{$L_c(h)=\scal{W_c^{\dag}h}{h}$}, where $A^{\dag}$ denotes the pseudoinverse of the operator $A$. However, the domain of $W_c^{\dag}$ is equal to the range of $W_c$, and so in general $K_x$ may not be in the domain of $W_c^{\dag}$. We will therefore introduce the orthogonal projection $W_c^{\dag}W_c$ mapping $\cH\mapsto\text{range}(W_c)$ and define the nonlinear control energy on $\cH$ as
\begin{equation}\label{eqn:best_lc}
L_c(h) = \scal{W_c^{\dag}(W_c^{\dag}W_c)h}{h}.
\end{equation}
We will consider finite sample approximations to~\eqref{eqn:best_lc}, however a further complication is that  $\widehat{W}_c^{\dag}\widehat{W}_c$ may not converge to $W_c^{\dag}W_c$ in the limit of infinite data (taking the pseudoinverse is not a continuous operation), and $\widehat{W}_c^{\dag}$ can easily be ill-conditioned in any event. Thus one needs to impose regularization, and we replace the pseudoinverse $A^{\dag}$ with a regularized inverse $(A + \lambda I)^{-1}, \lambda > 0$ throughout. We note that the preceding observations were also made in~\cite{RosascoDensity}. Intuitively,
regularization prevents the estimator from overfitting to a bad or unrepresentative sample
of data. We thus define the estimator \mbox{$\hat{L}_c:\cX\to\RR_+$} (that is, on the domain $\{K_x~|~ x\in\cX\}\subseteq\cH$) to be
\begin{equation}\label{eqn:rkhs_lc_def}
\hat{L}_c(x)=\tfrac{1}{2}\bigl\langle(\widehat{W}_c + \lambda I)^{-2}\widehat{W}_c K_x,K_x\bigr\rangle, \quad x\in\cX
\end{equation}
with infinite-data limit
\[
L_c^{\lambda}(x) = \tfrac{1}{2}\scal{(W_c + \lambda I)^{-2}W_c K_x}{K_x},
\]
where $\lambda > 0$ is the regularization parameter.

Towards deriving an equivalent but computable expression for $\hat{L}_c$ defined in terms of kernels, we
recall the sampling operator $\Sx$ of~\cite{SmaleIntegral} and its adjoint. Let $\bx = \{x_i\}_{i=1}^{m}$ denote a generic sample of $m$ data points. To $\bx$ we can associate the operators
\begin{alignat*}{4}
\Sx &: \cH    &\to &\, \RR^{m}, &\quad h &\in\cH &\mapsto&\, \bigl(h(x_1),\ldots,h(x_{m})\bigr)\\
\Sxs &:\RR^{m} &\to &\, \cH, &\quad c &\in\RR^{m} &\mapsto&\, \textstyle\sum_{i=1}^{m}c_iK_{x_i}\,.
\end{alignat*}
If $\bx$ is the collection of $m=Nq$ controllability samples, one can check that
$\widehat{W}_c = \tfrac{1}{m}\Sxs\Sx$ and $K_c=\Sx\Sxs$. Consequently,
\begin{align*}
\hat{L}_c(x) &=\tfrac{1}{2}\scal{(\tfrac{1}{m}\Sxs\Sx + \lambda I)^{-2}\tfrac{1}{m}\Sxs\Sx K_{x}}{K_{x}}\\
&=\tfrac{1}{2m}\scal{\Sxs(\tfrac{1}{m}\Sx\Sxs + \lambda I)^{-2}\Sx K_{x}}{K_{x}}\\
&= \tfrac{1}{2m}{\bf k_c}(x)^{\tr}(\tfrac{1}{m}K_c + \lambda I)^{-2}{\bf k_c}(x),
\end{align*}
where ${\bf k_c}(x):=\Sx K_x = \bigl(K(x,x_{\mu})\bigr)_{\mu=1}^{Nq}$ is the $Nq$-dimensional column vector
containing the kernel products between $x$ and the controllability samples.

Similarly, letting $\bx$ now denote the collection of $m=Np$ observability samples, we can approximate the future output energy by
\begin{align}
\hat{L}_o(x) &= \tfrac{1}{2}\bigl\langle\widehat{W}_oK_x,K_x\bigr\rangle  \\\label{eqn:Lo_rkhs}
&= \tfrac{1}{2m}\bigl\langle\Sxs\Sx K_x, K_x\bigr\rangle \nonumber \\
 &= \tfrac{1}{2m}{\bf k_o}(x)^{\tr}{\bf k_o}(x)
 = \tfrac{1}{2m}\nor{{\bf k_o}(x)}_2^2\nonumber
\end{align}
where ${\bf k_o}(x):=\bigl(K(x,d_{\mu})\bigr)_{\mu=1}^{Np}$ is the $Np$-dimensional column vector
containing the kernel products between $x$ and the observability samples.
We collect the above results into the following definition:
\begin{definition}\label{def:rkhs_energies}  Given a nonlinear control system of the form~\eqref{sigma}, we define the kernel
controllability energy function and the kernel observability energy function as, respectively,
\begin{align}
\hat{L}_c(x) &= \tfrac{1}{2Nq}{\bf k_c}(x)^{\tr}(\tfrac{1}{Nq}K_c + \lambda I)^{-2}{\bf k_c}(x) \\ \label{eqn:lc_hat}
\hat{L}_o(x) &= \tfrac{1}{2Np}\nor{{\bf k_o}(x)}_2^2 \;.
\end{align}
\end{definition}
Note that the kernels used to define $\hat{L}_c$ and $\hat{L}_o$ need not be the same.

\subsection{Consistency}
We'll now turn to showing that the estimator $\hat{L}_c$ is consistent, but note that
{\em we do not address the approximation error} between the energy function estimates and the true
 but unknown underlying functions. Controlling the approximation error requires making specific assumptions
about the nonlinear system, and we leave this question open.

In the following we will make an important set of assumptions regarding the kernel $K$ and the
RKHS $\cH$ it induces.
\begin{assumption}\label{ass:rkhs}
The reproducing kernel $K$ defined on the compact statespace $\cX\subset\RR^n$ is locally Lipschitz, measurable and defines a completely regular RKHS. Furthermore the diagonal of $K$ is uniformly bounded,
\[\label{eqn:kappa}
\kappa^2 = \sup_{x\in\cX}K(x,x) <\infty.
\]
\end{assumption}
Separable RKHSes are induced by continuous kernels on separable spaces $\cX$.
Since $\cX\subset\RR^n$ is separable and locally Lipschitz functions are also continuous, $\cH$ will always be separable. {\em Completely regular} RKHSes are introduced in~\cite{RosascoDensity} and
the reader is referred to this reference for details. Briefly, complete regularity ensures recovery of level sets of {\em any} distribution, in the limit of infinite data. The Gaussian kernel does not define a completely regular RKHS, but the $L_1$ exponential and Laplacian kernels do~\cite{RosascoDensity}.

We introduce some additional notation. Let $W_{c,m}$ denote the empirical RKHS gramian formed from a sample of size $m$ observations, and let the corresponding control energy estimate in Definition~\ref{def:rkhs_energies} involving $W_{c,m}$ and regularization parameter $\lambda$ be denoted by $L_{c,m}^{\lambda}$.

The following preliminary lemma provides finite sample error bounds for Hilbert-Schmidt covariance matrices
on real, separable reproducing kernel Hilbert spaces.
\begin{lemma}[\cite{RosascoIntegral} Theorem 7; Props. 8, 9]\label{lem:cov_conc}
\mbox{}
\begin{itemize}\itemsep 0pt
\item[(i)] The operators $W_c, W_{c,m}$ are Hilbert-Schmidt.
\item[(ii)] Let $\delta\in(0,1]$. With probability at least $1-\delta$,
\[
\nor{W_c-W_{c,m}}_{HS} \leq \frac{2\sqrt{2}\kappa^2}{\sqrt{m}}\log^{1/2}\frac{2}{\delta}.
\]
\end{itemize}
\end{lemma}

The following theorem establishes consistency of the estimator $L_{c,m}^{\lambda}$, the proof of which follows the method of integral operators developed by~\cite{SmaleIntegral,AndreaFastRates} and subsequently adopted in the context of density estimation by~(\cite{RosascoDensity}, Theorem 1).
\begin{theorem}
\mbox{}
\begin{itemize}\itemsep 0pt
\item[(i)] Fix $\lambda > 0$. For each $x\in\cX$, with probability at least $1-\delta$,
\[
\bigl|L_{c,m}^{\lambda}(x) - L_c^{\lambda}(x)\bigr| \leq \frac{2\sqrt{2}\kappa^4(\lambda^2 + \kappa^4)}{\lambda^4\sqrt{m}}\log^{1/2}\frac{2}{\delta} .
\]
\item[(ii)] If $(K,\cX,\mu_{\infty})$ is such that
    \[\label{eqn:bounded_pinv}
    \sup_{x\in\cX}\|W_c^{\dag}(W_c^{\dag}W_c)K_x\|_{\cH}<\infty,
    \]
then for all $x\in\cX$,
$$
\displaystyle\lim_{\lambda\to 0}|L_c^{\lambda}(x) -  L_c(x)|  = 0.
$$
\item[(iii)] If the condition~\eqref{eqn:bounded_pinv} holds and the sequence $\{\lambda_m\}_m$ satisfies $\displaystyle\lim_{m\to\infty}\lambda_m=0$ with
$\displaystyle\lim_{m\to\infty}\tfrac{\log^{1/2}m}{\lambda_m\sqrt{m}} = 0$, then
$$
\lim_{m\to\infty}\bigl|L_{c,m}^{\lambda}(x) - L_c(x)\bigr| = 0,\quad\text{almost surely.}
$$
\end{itemize}
\end{theorem}
\begin{proof}
For (i), the sample error, we have
\begin{align*}
2\bigl|L_{c,m}^{\lambda}(x) - L_c^{\lambda}(x)\bigr|
& \leq \nor{(W_{c,m} + \lambda I)^{-2}W_{c,m} - (W_c+ \lambda I)^{-2}W_c }\nor{K_{x}}^2_{\cH} \\
& \leq \bigl\|(W_{c} + \lambda I)^{-2}[\lambda^2(W_{c,m} - W_c) + W_c(W_c - W_{c,m})W_{c,m}](W_{c,m} + \lambda I)^{-2} \bigr\|\kappa^2\\
& \leq \frac{\kappa^2(\lambda^2 + \kappa^4)}{\lambda^4}\nor{W_{c,m} - W_c}_{HS}
\end{align*}
where $\nor{\cdot}$ refers to the operator norm. The second inequality follows from spectral calculus and~\eqref{eqn:kappa}. The third line follows making use of the estimates $\nor{(W_{c,m} + \lambda I)^{-2}}\leq \lambda^{-2}, \nor{(W_c + \lambda I)^{-2}}\leq \lambda^{-2}, \|W_c\|_{HS}\leq \kappa^2, \|W_{c,m}\|_{HS}\leq\kappa^2$ (and the fact that $\lambda > 0$ so that the relevant quantities are invertible).
Part (i) then follows applying Lemma~\ref{lem:cov_conc} to the quantity
$\nor{W_{c,m} - W_c}_{HS}$.
For (ii), the approximation error, note that the compact self-adjoint operator $W_c$ can be expanded onto an orthonormal basis $\{\sigma_i,\phi_i\}$. We then have
\begin{align*}
 2\bigl|L_{c}^{\lambda}(x) - L_c(x)\bigr|
& = \bigl|\bigl\langle[(W_{c} + \lambda I)^{-2}W_{c} - W_c^{\dag}(W_c^{\dag}W_c)]K_x,K_x\bigr\rangle\bigr| \\
& = \left| \sum_i\frac{\sigma_i}{(\sigma_i + \lambda)^2}|\langle \phi_i,K_x\rangle|^2 -
\sum_{i:\sigma_i>0}\frac{1}{\sigma_i}|\langle\phi_i,K_x\rangle|^2\right| \\
& \leq \lambda\sum_{i:\sigma_i>0}\frac{2\sigma_i + \lambda}{(\sigma_i + \lambda)^2\sigma_i}|\langle \phi_i,K_x\rangle|^2 .
\end{align*}
The last quantity above can be seen to converge to 0 as $\lambda\to 0$ since the sum converges for all $x$ under the condition~\eqref{eqn:bounded_pinv}.
Lastly for part (iii), we see that if $m\to\infty$ and $\lambda^2\to 0$ slower than $\sqrt{m}$ then the sample error (i) goes to 0 while (ii) also holds. For almost sure convergence in part (i), we additionally require that for any $\varepsilon\in(0,\infty)$,
$$
\sum_m\bbP\bigl(|L_{c,m}^{\lambda}(x) - L_c^{\lambda}(x)| > \varepsilon\bigr)\leq
\sum_m e^{-\cO(m\lambda^4_m\varepsilon^2)} < \infty.$$
The choice $\lambda_m = \log^{-1/2}m$ satisfies this requirement, as can be seen from the fact that for large enough $M<\infty$, $\sum_{m>M} e^{-m/\log^2 m} \leq \sum_{m>M} e^{-\sqrt{m}} < \infty$.
\end{proof}
We note that the condition~\eqref{eqn:bounded_pinv} required in part (ii) of the theorem has also been discussed in the context of support estimation in forthcoming work from the authors of~\cite{RosascoDensity}.

\subsection{Comments on the Approximation Errors for the Observability Energy}

In this section,  we review some existing results for the approximation errors for Lyapunov functions using kernel methods.  In \cite{giesl_book, GieWen07}, the authors study the problem of approximate solutions using Radial Basis Functions (RBFs) of a general linear PDE
\[Lu=v \; \mbox{on}\; \Omega, \]
where $\Omega$ is a domain in $\RR^n$ and $L$ is a linear differential operator of the form
\[Lu(x)=\sum_{|\alpha| \le m}c_{\alpha}(x)D^{\alpha}u(x), \]
where $c_{\alpha} \in C^{\sigma}(\bar{\Omega},\RR)$, i.e. the derivatives of order $\beta$ with $|\beta| \le \sigma$ exist and are continuous on $\bar{\Omega}$. They particularly considered the case of orbital derivatives of a function $u$ with respect to the ODE $\dot{x}=f(x)$, i.e. when
\[Lu(x) :=\langle \nabla u(x), f(x) \rangle = \sum_{j=1}^n f_j(x) \partial_ju(x) \] 

\begin{theorem} \label{thm_v1} Consider $\dot{x}=f(x)$ with $f \in C^{\sigma}(\RR^n,\RR^n) $ and let $x_0$ be an equilibrium such that all eigenvalues of $Df(x_0)$ have a negative real part.
Let $p(x) \in C^{\sigma}(\RR^n,\RR)$ satisfy the following conditions: a.) $p(x)>0$ for $x \ne x_0$, b.) $p(x)=O(||x-x_0||_2^{\eta}) $ with $\eta >0$ for $x \rightarrow x_0$, c.) For all $\epsilon >0$, $p$ has a lower positive bound on $\RR^n \setminus B(x_0,\epsilon)$ where $B(x_0,\epsilon)$ is a the ball centered at $x_0$ of radius $\epsilon$. Then there exists a  Lyapunov function $V_1 \in C^{\sigma}(A(x_0),\RR)$ such that $V_1(x_0)=0$ and 
\[ \label{pde_v1}LV_1(x)=f_1(x) := -p(x), \; \mbox{\; for all \;} x \in A(x_0),\]
where $A(x_0)$ is the basin of attraction of $x_0$.
\end{theorem}

This theorem shows that by solving the equation (\ref{pde_v1}) one can find an approximation $V_1$ of the Lyapunov function $V$ for the ODE $\dot{x}=f(x)$ that satisfies $LV(x) < 0$. Moreover, the authors of  \cite{giesl_book, GieWen07} provided with error estimates when the solution of (\ref{pde_v1}) is obtained using a special kind of kernels called Wendland functions. The authors adopted a more direct approach in solving (\ref{pde_v1}) that we plan on using to find an alternative approximation of the controllability energy $L_c$. The advantage of using Wendland functions as kernels is that the resulting RKHS is also a Sobolev space thus allowing to use Sobolev inequalities to get estimates on the approximation of the solution of the Lyapunov PDE, cf. \cite{wendland} for more details about Wendland functions and the properties of the RKHSes associated to them. 

Briefly, let $\Phi(x) =\psi_k(||x||)$ be a radial function where $\psi_k$ is a Wendland function. Consider the grid points $X_N=\{x_1,\cdots,x_N \} \subset \RR^n$. Consider the following ansatz 
\[\label{betas}V_1(x)=\sum_{k=1}^n \beta_k (\delta_{x_k}\circ L)^y\Phi(x-y),\]
where $(\delta_{x_k}\circ L)^y$ denotes differentiation with respect to $y$ then evaluation at $y=x_k$.
By considering the interpolation conditions \[LV_1(x_j)=LV(x_j)=f_1(x_j),\]
and by plugin in the ansatz
\[\sum_{i=1}^N \beta_k  \underbrace{(\delta_{x_j}\circ L)^x (\delta_{x_k}\circ L)^y\Phi(x-y)}_{=a_{jk}}=LV(x_j)=f_1(x_j)=:\gamma_j,\] 
one gets a system of linear algebraic equations for the $\beta$ in (\ref{betas}):
\[A \beta = \gamma, \]
 where the matrix $A$ is symmetric and positive definite. For an approximation $V_1$ of the Lyapunov function $V$, the authors of \cite{giesl_book, GieWen07} proved
 \begin{theorem} \cite{giesl_book, GieWen07} Let $\psi_k$, $k \in \NN$, be a Wendland function and let $\Phi(x)=\psi_k(||x||) \in C^{2k}(\RR^n,\RR)$ be a radial basis function.  Let $f \in C^{\sigma}(\RR^n,\RR)$ where $\sigma \ge \frac{n+1}{2}+k$. Then, for each compact set $K_0 \subset A(x_0)$ there is $C^{\ast}$ such that
 \[\label{errors} |V^{\prime}(x)-V_1^{\prime}(x) | \le C^{\ast} h^{\theta} \; \mbox{for all}\; x \in K_0, \]
 where $h:=\max_{y \in K_0}\min_{x \in X_n}||x-y||$ is the fill distance and $\lambda=1/2$ for $k=1$ and $\lambda=1$ for $k \ge 2$ (or $\lambda=k-1/2$ as proved in \cite{GieWen07}).
 
 \end{theorem} 

Given the above results, we could immediately derive error estimates for the observability energy $L_o$ that satisfies (\ref{Lo_hjb}). If $h(x)$ is such that $p(x)=\frac{1}{2}h^T(x)h(x)$  satisfies the conditions in Theorem \ref{thm_v1}, then we immediately get an expression of $\hat{L}_o$ given by (\ref{betas}) and error estimates given in (\ref{errors}).  We leave for future work the extension of such approach when $\frac{1}{2}h^T(x)h(x)$ does not satisfy the conditions   in Theorem \ref{thm_v1} but we expect similar error estimates by performing a Taylor series expansion of $h$ and expressing $\frac{1}{2}h^T(x)h(x)=p(x)+e(||x-x_0||_2^{\eta+1})$ where $e$ is the error due to the Taylor expansion.

\section{Estimation of Invariant Measures for Ergodic Nonlinear SDEs}\label{sec:nonlinear-sdes}
In this Section we consider {\em ergodic} nonlinear SDEs of the form~\eqref{sde_nonlin}, where the invariant (or ``stationary'') measure is a key quantity providing a great deal of insight. 
Solving a Fokker-Planck equation
 of the form~\eqref{ito} is one way to determine the probability distribution describing the solution to an
SDE. However, for nonlinear systems finding an explicit solution to the Fokker-Planck equation --or even its
steady-state solution-- is a challenging problem. The study of existence of steady-state solutions can be traced back to the 1960s~\cite{fuller,zakai}, however explicit formulas for steady-state solutions of the Fokker-Planck equation exist in only a few special cases (see~\cite{butchart,da_prato1992,fuller, guinez,liberzon,risken}
for example). Such systems are often conservative or second order vector-fields. Hartmann~\cite{hartmann:08} among others has studied balanced truncation in the context of linear SDEs, where empirical estimation of gramians plays a key role.

We propose here a data-based non-parametric estimate of the solution to the steady-state Fokker-Planck equation~\eqref{steady} for a nonlinear SDE, by combining the relation~\eqref{rho_approx} with the control energy estimate~\eqref{eqn:lc_hat}. Following the general theme of this paper, we make use of the theory from the linear Gaussian setting described in Section~\ref{sec:linear-sdes}, but in a suitable reproducing kernel
Hilbert space. Other estimators have of course been proposed in the literature for approximating invariant measures and for density estimation from data more generally (see e.g.~\cite{biau,froyland,froyland1,kilminster,RosascoDensity}), however to our knowledge we are not aware of any estimation techniques which combine RKHS theory and nonlinear dynamical control systems. An advantage of our approach over other non-parametric methods is that an invariant density is approximated by way of a regularized fitting process, giving the user an additional degree of freedom in the regularization parameter.

Our setting adopts the perspective that the nonlinear stochastic system~\eqref{sde_nonlin} behaves
approximately linearly when mapped via $\Phi$ into the RKHS $\cH$, and as such
may be modeled by an infinite dimensional linear system in $\cH$. Although this system is {\em unknown},
we know that it is linear and that we can estimate its gramians and control energies from observed data. Furthermore, we know that the invariant measure of the system in $\cH$ is zero-mean Gaussian with covariance given by the controllability gramian. Thus the original nonlinear system's invariant measure on $\cX$ should be reasonably approximated by the pullback along $\Phi$ of the Gaussian invariant measure associated with the linear infinite dimensional SDE in $\cH$.

We summarize the setting in the following {\em modeling Assumption}:
\begin{assumption}\label{ass:ou_proc}
Let $\cH$ be a real, possibly infinite dimensional RKHS satisfying Assumption~\ref{ass:rkhs}.
\begin{itemize}
\item[(i)] Given a suitable choice of kernel $K$, if the $\RR^d$-valued stochastic process $x(t)$ is a solution to the (ergodic) stochastically excited nonlinear system~\eqref{sde_nonlin}, the $\cH$-valued stochastic process \mbox{$(\Phi\circ x)(t)=:X(t)$} can be reasonably modeled as an Ornstein-Uhlenbeck process
\[\label{eqn:infdim-sde}
dX(t) = AX(t)dt + \sqrt{C}dW(t), \quad X(0)=0\in\cH
\]
where $A$ is linear, negative and is the infinitesimal generator of a strongly continuous semigroup $e^{tA}$, $C$ is linear, continuous, positive and self-adjoint, and $W(t)$ is the cylindrical Wiener process.
\item[(ii)] The measure $P_{\infty}$ is the invariant measure of the OU process~\eqref{eqn:infdim-sde} and $P_{\infty}$ is the pushforward along $\Phi$ of the unknown invariant measure $\mu_{\infty}$ on the statespace $\cX$ we would like to approximate.
\item[(iii)] The measure $\mu_{\infty}$ is absolutely continuous with respect to Lebesgue measure, and so admits a density.
\end{itemize}
\end{assumption}
We will proceed in deriving an estimate of the invariant density under these assumptions, but note that
there are interesting systems for which  the assumptions may not always hold in practice. For example,  uncontrollable systems may not have a unique invariant measure. In these cases one must interpret the results discussed here as heuristic in nature.

It is known that a mild solution $X(t)$ to the SDE~\eqref{eqn:infdim-sde} exists
and is unique (\cite{da_prato1992}, Thm. 5.4. pg. 121). Furthermore, the controllability gramian
associated to~\eqref{eqn:infdim-sde}
\[\label{eqn:infdim-gramian}
W_ch = \int_0^{\infty}e^{tA}Ce^{tA^*}hdt,\quad h\in\cH
\]
is trace class (\cite{da_prato2006}, Lemma 8.19), and the unique measure $P_{\infty}$ invariant with respect to the Markov semigroup associated to the OU process has characteristic function (\cite{da_prato2006}, Theorem 8.20)
\[\label{eqn:inv-meas}
\widetilde{P}_{\infty}(h) = \exp\Bigl(-\tfrac{1}{2}\scal{W_ch}{h}\Bigr),\quad h\in\cH \;.
\]
We will use the notation $\widetilde{P}$ to refer to the Fourier transform of the measure $P$.
The law of the solution $X(t)$ to problem~\eqref{eqn:infdim-sde} given initial condition $X(0)=0$ is
Gaussian with zero mean and covariance operator $Q_t=\int_{0}^t e^{sA}Ce^{sA^*}ds$. Thus
\begin{align*}
W_c &= \lim_{t\to\infty}\bbE[X(t)\otimes X(t)]\\
 & = \int_{\cH}\scal{\cdot}{h}{h}dP_{\infty}(h)\\
&= \int_{\cX}\scal{\cdot}{K_x}K_xd\mu_{\infty}(x)
\end{align*}
where the last integral follows pulling $P_{\infty}$ back to $\cX$ via $\Phi$, establishing
the equivalence between~\eqref{eqn:infdim-gramian} and ~\eqref{eqn:covop_gramian}.

Given that the measure $P_{\infty}$ has Fourier transform~\eqref{eqn:inv-meas} and by Assumption~\ref{ass:ou_proc} is interpreted as the pushforward of $\mu_{\infty}$ (that is, for Borel sets $B\in\mathcal{B}(\cH)$, $P_{\infty}(B)=(\Phi_*\mu_{\infty})(B)=\mu_{\infty}(\Phi^{-1}(B))$ formally), we
have that $\widetilde{\mu}_{\infty}(x) = \exp\bigl(-\tfrac{1}{2}\scal{W_cK_x}{K_x}\bigr)$.

The invariant measure $\mu_{\infty}$ is defined on a finite dimensional space, so together with
part (iii) of Assumption~\ref{ass:ou_proc}, we may consider the corresponding (Radon-Nikodym) density
$$\rho_{\infty}(x) \propto \exp\bigl(-\tfrac{1}{2}\scal{W_c^{\dag}(W_c^{\dag}W_c)K_x}{K_x}\bigr)$$
whenever the condition~\eqref{eqn:bounded_pinv} holds. If~\eqref{eqn:bounded_pinv} does not hold or if we are considering a finite data sample, then we regularize to arrive at
\[
\rho_{\infty}(x) \propto \exp\bigl(-\tfrac{1}{2}\scal{(W_c + \lambda I)^{-1}K_x}{K_x}\bigr)
\]
as discussed in Section~\ref{sec:linear-sdes} (see Eq.~\ref{rho_approx}) and Section~\ref{sec:energy_fns}. This density may be
estimated from data $\{x_i\}_{i=1}^N$ since the controllability energy may be estimated from data: at a new point $x$, we have
\[
 \hat{\rho}_{\infty}(x) = Z^{-1}\exp\bigl(-\hat{L}_c(x)\bigr)
\]
 where $\hat{L}_c$ is the empirical approximation computed according to Definition~\ref{def:rkhs_energies},
and the constant $Z$ may be either computed analytically in some cases or simply estimated from the data
sample to enforce summation to unity.
We may also estimate, for example, level sets of $\rho_{\infty}$ (such as the support) by considering level
sets of the regularized control energy function estimator,  $\{x\in\cX~|~ L_{c,m}(x) \leq \tau\}$.

%

\subsection{Numerical Examples}
Consider the SDE $dX=-5X^5+10X^3+\sqrt{2} dW$. This is a gradient system $dX=-\nabla \Phi(x)+b dW $ and the exact stationary measure is given by $ \rho_{\infty}(x)=Ne^{-2\phi(x)/b^2}$. The figure below shows the the comparison between the exact steady-state measure, our estimate and the empirical estimate (obtained directly by counting the data points).

\begin{figure}
 \includegraphics[height=15cm]{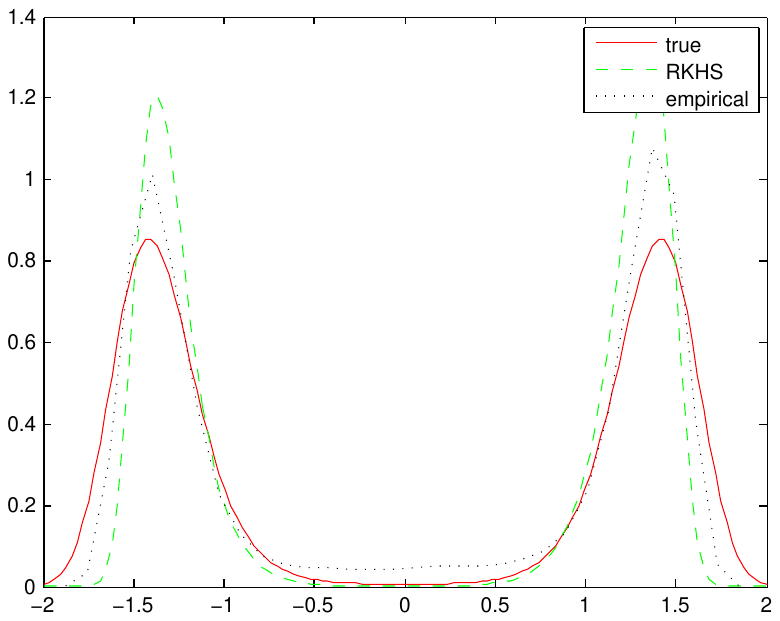}
\end{figure}

\section{Conclusion}
To summarize our contributions, we have introduced estimators for the controllability/observability energies  of nonlinear control systems. We showed that the controllability energy
estimator may be used to approximate the stationary solution of the Fokker-Planck equation governing nonlinear SDEs.

The estimators we  derived were based on applying linear methods for control and random
dynamical systems to nonlinear control systems and SDEs, once mapped into an
infinite-dimensional RKHS acting as a ``linearizing space''. These results collectively argue that working in reproducing kernel Hilbert spaces  offers tools for a data-based theory for nonlinear dynamical systems.

We leave for future work the formulation of data-based estimators for Lyapunov exponents and the controllability/observability operators $\Psi_c,\Psi_o$ associated to nonlinear systems.

\section*{Acknowledgements}
We  thank Lorenzo Rosasco and Jonathan Mattingly for helpful discussions.  BH thanks the European Commission and the
Scientific and the Technological Research Council of Turkey (Tubitak)
for financial support received through a Marie Curie Fellowship, and JB gratefully acknowledges support under NSF contracts NSF-IIS-08-03293 and NSF-CCF-08-08847 to M. Maggioni.

\end{document}